\documentclass[11pt, reqno,amscd, psamsfonts]{amsart}
\usepackage{amsthm, amsmath}
\usepackage{amssymb}
\usepackage{amscd}
\usepackage{graphicx}
\usepackage{amsfonts}
\usepackage{color}

\usepackage{marginnote}

\newcounter{mynote}

\newtheorem{thm}        {Theorem}     [section]

\newtheorem{lemma}  [thm]  {Lemma}
\newtheorem{cor}    [thm]  {Corollary}

\newtheorem{theorem}      [thm]   {Theorem}

\newtheorem{corollary}    [thm]  {Corollary}
\newtheorem{conjecture}   [thm]  {Conjecture}

\theoremstyle{remark}
\newtheorem*{remark}{Remark}

\newtheorem{problem}{Problem}

\newtheorem*{CC}{{\rm{\bf Theorem C}}}

\def\Z{{\mathbb Z}}

\def\PP1{{\mathb P}^1}

\DeclareMathOperator\aut{Aut}

\DeclareMathOperator\PSL{PSL}
\DeclareMathOperator\PGL{PGL}
\DeclareMathOperator\SL{SL}
\DeclareMathOperator\GL{GL}

\DeclareMathOperator\Spec{Spec}
\DeclareMathOperator\Oo{O}

\renewcommand\b\bar

\newcommand\Cp{\mathcal{C}_p}
\newcommand\Sz{\mathrm{Sz}}
\newcommand\Ree{\mathrm{Re}}
\newcommand\PSU{\mathrm{PSU}}

\begin{document}

\title[Global Oort Groups]
{Global Oort Groups}
\author{Ted Chinburg}
\address{Department of Mathematics,  University of Pennsylvania,
Philadelphia, PA 19104-6395, USA}
\email{ted@math.upenn.edu}
\author{Robert Guralnick}
\address{Department of Mathematics,  University of Southern
California,  Los Angeles,  CA 90089-2532,  USA}
\email{guralnic@usc.edu}
\author{David Harbater}
\address{Department of Mathematics,  University of Pennsylvania,
Philadelphia PA 19104-6395, USA}
\email{harbater@math.upenn.edu}

\thanks{The authors were supported by NSF FRG grant DMS-1265290.
The first author was also supported by NSF FRG DMS-1360767, 
NSF SaTC grant CNS-1513671 and Simons Foundation Grant 338379.
The second author was also supported by NSF DMS-1265297 and NSF DMS-1302886.
The third author was also supported by NSA grant H98230-14-1-0145 
and NSF FRG grant DMS-1463733.}

\subjclass[2010]{Primary 20B25, 12F10,  14H37;
Secondary l3B05,  14D15,  14H30}

\keywords{Curves,  covers, automorphisms,  Galois groups,  characteristic $p$,
lifting,  Oort conjecture, simple group, almost simple, normal complement}

\date{December 30, 2015}

\begin{abstract}
We study the Oort groups for a prime $p$, i.e.\ finite groups $G$
such that every $G$-Galois branched cover of smooth curves over
an algebraically closed field of characteristic $p$
lifts to a $G$-cover of curves in characteristic $0$. 
We prove that all Oort groups lie in a particular class of finite groups
that we characterize, with equality of classes under a conjecture about local 
liftings.  We prove this equality unconditionally if the order of $G$ 
is not divisible by $2p^2$.  We also treat the local lifting problem 
and relate it to the global problem.
\end{abstract}

\maketitle

\section{Introduction} \label{intro sect}

This paper, which is a sequel to \cite{cgh1} and \cite{cgh2}, concerns the question of when covers of curves in characteristic $p>0$ lift to covers in characteristic zero.
We begin with a $G$-Galois finite branched cover $X \to Y$ of smooth projective curves over an algebraically closed field $k$ of characteristic $p$.
By a {\em lifting} we mean 
a $G$-Galois branched cover $\mathcal X \to \mathcal Y$ of normal projective curves over a complete discrete valuation ring $R$ such that the following is true.  The fraction field of $R$ has characteristic zero, the residue field of $R$ is isomorphic to $k$, and the closed fiber of $\mathcal X \to \mathcal Y$ is $G$-isomorphic to the given cover $X \to Y$.  Lifting a $G$-Galois cover $X \to Y$ is equivalent to lifting the action of $G$ on $X$ to an action of $G$ on $\mathcal X$.

In \cite{gro}, Grothendieck showed that all tamely ramified covers lift, and in particular a cover lifts if the Galois group has order prime to $p$.
Later, in \cite[Sect.~I.7]{oort}, F.~Oort posed the following conjecture:

\begin{conjecture} \label{Oort conj}
(Oort Conjecture)
Every cyclic Galois cover of $k$-curves lifts to characteristic zero.
\end{conjecture}

Motivated by this conjecture, in \cite{cgh1} we defined a finite group $G$ to be a (global) {\em Oort group} with respect to $k$ if every $G$-Galois cover of smooth projective $k$-curves lifts to characteristic zero. We asked which finite groups are Oort groups.
Recently, Conjecture~\ref{Oort conj} was proven by the combined work of  Obus and Wewers \cite{ow} and Pop \cite{pop},  so we now know that every cyclic group is an Oort group.  A brief history is as follows:
In \cite{OSS} and later in \cite{GM}, it was shown that a cyclic group $G$ is an Oort group provided that its order is exactly divisible by $p$, respectively by $p^2$.
More recently, in \cite{ow}, it was shown that a cyclic group $G$ is an Oort group provided that its order is divisible at most by $p^3$, or more generally if its higher ramification groups satisfy a certain condition.
Finally, by deforming a general cover to one in the special case treated in \cite{ow}, the full conjecture was proven in \cite{pop}.

While the above conjecture considered only cyclic groups, some results have been obtained about more general groups.
Apart from Grothendieck's result on prime-to-$p$ groups being Oort groups, it is known that the dihedral group $D_{2p}$ of order $2p$ is an Oort group in characteristic $p$ (see \cite{pagot} for $p=2$, \cite{BW} for $p$ odd), and
$A_4$ is an Oort group in characteristic $2$ (see \cite{BW}).  Note that these groups are cyclic-by-$p$, i.e. they are extensions of a cyclic group of order prime to $p$ by a normal $p$-Sylow subgroup.
The significance of understanding Oort groups of this form is that a group $G$ is an Oort group if and only if each cyclic-by-$p$ subgroup of $G$ is an Oort group \cite[Corollary~2.8]{cgh1}.

In \cite[Corollary~3.4, Theorem~4.5]{cgh1}, we proved that if a cyclic-by-$p$ group $G$ is an Oort group for an algebraically closed field of characteristic $p$, then $G$ is either a cyclic group $C_m$ of order $m$, a dihedral group $D_{2p^n}$ for some $n$, or $A_4$ if $p=2$.  We also conjectured
the converse in \cite{cgh1}:

\begin{conjecture} \label{strong Oort conj}
(Strong Oort Conjecture)
If $k$ is an algebraically closed field of characteristic $p$,
then a cyclic-by-$p$ group $G$ is an Oort group for $k$ if and only if $G$ is of the form
$C_m$, $D_{2p^n}$, or $A_4$ (the last case only if $p=2$).
\end{conjecture}

Motivated by this, we define a group $G$ to be an {\em $\Oo$-group} for the prime $p$ if every cyclic-by-$p$ subgroup of $G$ is of the form $C_m$, $D_{2p^n}$, or $A_4$ (if $p=2$, in this last case).  
By \cite[Corollaries~3.5 and~4.6]{cgh1}, every Oort group for an algebraically closed field $k$ of characteristic $p$ is an $\Oo$-group for $p$.   

In this paper we consider the converse to this last assertion, i.e.\ whether every $\Oo$-group is an Oort group.  We show that this implication would follow from Conjecture~\ref{strong Oort conj}, 
and that it holds
unconditionally for groups whose orders are not divisible by $2p^2$ (see 
Theorem~\ref{dihedral implies SOC} and Corollary~\ref{Oort O implications}).
We also give a classification of $\Oo$-groups, in Theorems~\ref{main p odd}, \ref{Oort 2 Sylow},
and~\ref{main2}, using finite group theory, including the classification theorem of finite simple groups.  If Conjecture~\ref{strong Oort conj} is later proven to hold in general, this would fully classify Oort groups for $k$ and show that the set of Oort groups depends only on $p$ (as has been conjectured).
Unconditionally, our classification of $\Oo$-groups restricts the set of finite groups that can be Oort groups in characteristic $p$.

We also consider the companion notation of ``local Oort groups'', i.e.\
groups for which every local cover lifts.  More precisely, $G$ is a {\em local Oort group} for $k$ if every connected normal $G$-Galois branched cover of $\Spec(k[[t]])$ lifts to such a cover of $\Spec(R[[t]])$, for some complete discrete valuation ring $R$ whose fraction field has characteristic zero and whose residue field is $k$.  Since Galois groups over $k((t))$ are all cyclic-by-$p$, so are all local Oort groups for $k$.  In \cite[Theorem~2.4]{cgh1}, it was shown that $G$ is a global Oort group for $k$ if and only if every cyclic-by-$p$ subgroup of $G$ is a local Oort group for $k$; and every cyclic-by-$p$ Oort group for $k$ is a local Oort group for $k$ \cite[Corollary~2.6]{cgh1}.  In Theorem~\ref{local global equiv} of the present paper we prove the converse of this last assertion; thus local Oort groups are the same as cyclic-by-$p$ global Oort groups.  Hence local Oort groups are preserved under taking subquotients; the corresponding property of {\em global} Oort groups was shown in \cite[Corollary~2.7]{cgh1}.  In this way, Conjecture~\ref{strong Oort conj} can  be regarded as a conjecture about local liftings.

\smallskip

The paper is organized as follows.  In Section~\ref{main results sect},
we state the main results about $\Oo$-groups, global Oort groups, and local Oort groups.
The proof of the classification of $\Oo$-groups in odd characteristic $p$ is given in Section~\ref{odd char sect}, and the proof in characteristic $2$ is given in Section~\ref{char 2 sect}.
We remark that the proof for
odd $p$ requires the classification of finite simple groups.
The proof for $p=2$ requires only an older result about
the classification of finite simple groups with a dihedral Sylow
$2$-subgroup and the Feit-Thompson theorem.

\smallskip

\noindent \textbf{Notation and Terminology}

\smallskip

In this paper, $k$ is an algebraically closed field of characteristic $p \ne 0$.  {\em Curves} over $k$ are smooth, connected projective $k$-schemes of dimension one.  For a discrete valuation ring $R$, an {\em $R$-curve} is a separated flat projective $R$-scheme whose fibers are curves.  A {\em $G$-Galois cover} is a finite generically separable morphism of connected normal schemes $X \to Y$ together with a faithful action of $G$ on $X$ such that $Y = X/G$.

Given a group $G$ with subgroups $N,H$, we write $G=NH$ if 
$H$ normalizes $N$ and $G$ is the set of products $nh$ with $n\in N$ and $h \in H$.  This is a semi-direct product if and only if $N \cap H = 1$.  We write $C_m$ for the cyclic group of order $m$, and $D_{2n}$ for the dihedral group of order $2n$.

We have from \cite{asch} some basic notions and results about finite groups.  See also \cite[2.5]{gls3} as a general  reference about automorphism groups
of simple groups, e.g.\ in connection with Theorem~\ref{simple-odd}.

A {\em section} (or {\em subguotient}) of a group $G$ is a group $H/K$ where $K$ is normal
in $H$ and $H \le G$. A {\em chief factor} of a group $G$
is a non-trivial section $H/K$ where $H,K$ are both normal in $G$
and there are no normal subgroups of $G$ properly between
$H$ and $K$.  
If moreover $H$ is contained in some subgroup $E \subset G$, we also say that $H/K$ is a 
$G$-{\em chief factor} of $E$. 

A subgroup $H$ of $G$ is {\it subnormal} in $G$ if there is a finite chain
of subgroups $H = H_0 \subset H_1 \cdots \subset H_n = G$ such that $H_i$ is normal in $H_{i+1}$
for all $i$.  The {\em automizer} of a subgroup $H \le G$ is $N_G(H)/C_G(H)$, where
$N_G(H)$, $C_G(H)$ denote the normalizer and centralizer of $H$ in $G$. Let $Z(G)$ be the
center of $G$.

Given a set of primes $\pi$, write $\pi'$ for the complement of $\pi$.
If $G$ is a finite group,  let $\pi(G)$ denote the set of
prime divisors of the order $|G|$ of $G$.   We say $G$ is a {\em $\pi$-group} if
$\pi(G)  \subseteq \pi$.
Let $O_{\pi}(G)$ be the largest normal $\pi$-subgroup of $G$.
We write
$O_p(G)$ for $O_{\{p\}}(G)$ and write
$O(G)$ for $O_{2'}(G)$.
If $P$ is a $p$-group,  then $\Omega_1(P)$ is the subgroup
of $P$ generated by all elements of order $p$.

A group $G$ is {\em perfect} if it is equal to its commutator subgroup $[G,G]$; and it
is {\em quasisimple} if it is perfect and $G/Z(G)$ is simple.
A {\em component} of a finite group is a subnormal quasisimple subgroup.
It is known that any two distinct components commute.  We write $E(G)$ for the subgroup of
$G$ generated by all components.

The {\em Fitting subgroup},  $F(G)$,  is the largest normal nilpotent
subgroup of $G$.  Equivalently, $F(G)$ is the direct product of
$O_p(G)$, as $p$  ranges over $\pi(G)$.
The  {\em generalized Fitting subgroup} of $G$ is the group
$F^*(G):=E(G)F(G)$.
This group satisfies
the important property that $C_G(F^*(G))=Z(F(G))$.
A group $G$ is called {\em almost simple} if $F^*(G)$ is a non-abelian
simple group  (or equivalently,  $S \le G \le \aut(S)$ where
$S$ is a non-abelian simple group).

\section{Main Results} \label{main results sect}

We fix an algebraically closed field of characteristic $p$, and consider Oort groups for $k$.
Cyclic-by-$p$ Oort groups are local Oort groups by \cite{cgh1}, 
and Oort groups are $\Oo$-groups by \cite[Corollaries~3.5 and~4.6]{cgh1}.
Using the (now-proven) Oort Conjecture together with additional results (especially from \cite{cgh1}), we prove the converse of the first fact and a partial converse of the second.
We also give an explicit description of $\Oo$-groups; this restricts the forms that Oort groups can have, and would classify those groups if it is shown that Oort groups and $\Oo$-groups are the same in general.

\begin{thm} \label{local global equiv}
A cyclic-by-$p$ group is an Oort group if and only if it is a local Oort group.
\end{thm}

\begin{proof}
The forward direction was shown at \cite[Corollary~2.6]{cgh1}.  For the reverse direction, let $G$ be a local Oort group.

If $p \ne 2$, then $G$ is either cyclic or is dihedral of order $2p^n$ for some $n$ by \cite[Theorem~3.3]{cgh1}.  In the former case, $G$ is an Oort group by \cite{pop}.  In the latter case, every quotient of $G$ is a local Oort group, by \cite[Proposition~2.11]{cgh1}.
Hence $D_{2p^m}$ is a local Oort group for all $m \le n$.  But also every cyclic group is an Oort group by \cite{pop}, and hence is a local Oort group by the forward direction of the theorem.  Thus
every subgroup of $G$ is a local Oort group.  It then follows from \cite[Theorem~2.4]{cgh1} that $G$ is an Oort group.

Next suppose that $p=2$.  By \cite[Theorem~4.4]{cgh1}, $G$ is either cyclic, or is a dihedral
$2$-group, or is the alternating group $A_4$, or is a semi-dihedral or generalized quaternion group of order at least $16$.  In the first two cases, the conclusion follows as for the case of $p$ odd.  If $G = A_4$, then every proper subgroup of $G$ is either cyclic or a Klein four group.  But these subgroups are Oort groups by \cite{pop} and \cite{pagot} respectively, and hence are local Oort groups as above.  So $G$ is an Oort group by \cite[Theorem~2.4]{cgh1}.  As for the remaining two cases, in fact they do not occur.  Namely,
semi-dihedral groups of order at least $16$ are not local Oort groups, because 
\cite[Theorem~4.1]{cgh2} says that 
every local Oort group is a KGB-group (see \cite[Definition~1.1]{cgh2} for the definition of that notion), and because such semi-dihedral groups are not KGB groups by \cite[Theorem~1.2]{cgh2}.  Moreover generalized quaternion groups are also not local Oort groups, by \cite[Proposition~4.7, Theorem~4.8]{bw}.
\end{proof}

Hence the set of local Oort groups for $k$ is closed under taking subquotients, since this holds for the set of Oort groups for $k$ \cite[Corollary~2.7]{cgh1}.  (In \cite[Proposition~2.11]{cgh1} it was shown the local Oort groups are closed under taking quotients; but it had been left open whether closure also held for subgroups.)

The next two results give a partial converse to the fact (see
\cite[Corollaries~3.5 and~4.6]{cgh1}) that every Oort group is an $\Oo$-group.

\begin{thm} \label{dihedral implies SOC}
Let $k$ be an algebraically closed field of characteristic $p$.
\begin{enumerate}
\item \label{D2pn implies weak SOC}
Let $n$ be a positive integer, and suppose that the dihedral group $D_{2p^n}$ is an Oort group for $k$.
Then for every finite group $G$ of order not divisible by $2p^{n+1}$, $G$ is an Oort group for $k$ if and only if $G$ is an $\Oo$-group for $p$.
\item \label{SOC special case}
In particular, if $2p^2$ does not divide the order of a finite group $G$, then $G$ is an Oort group for $k$ if and only if $G$ is an $\Oo$-group for $p$.
\end{enumerate}
\end{thm}

\begin{proof}
Every Oort group is an $\Oo$-group by  
\cite[Corollaries~3.5 and~4.6]{cgh1}, so it suffices to show the reverse 
implication, under the hypothesis of the theorem.  By \cite[Theorem~2.4]{cgh1}, it is enough to 
show that every cyclic-by-$p$ subgroup $H$ of $G$ is a local Oort group, or equivalently (by 
Theorem~\ref{local global equiv} above) that $H$ is an Oort group.

Let $G$ be $\Oo$-group for $p$ whose order is not divisible by $2p^{n+1}$.  Then every cyclic-by-$p$ subgroup of $G$ is either cyclic, or of the form $D_{2p^m}$ for some $m \le n$, or (if $p=2$) isomorphic to $A_4$.  
Cyclic groups are Oort groups by \cite{ow} and \cite{pop}, and $A_4$ is an Oort group for $p=2$ by \cite{BW}.
By hypothesis, $D_{2p^n}$ is an Oort group, and hence
so is its quotient group $D_{2p^m}$ for $m \le n$, by \cite[Corollary~2.7]{cgh1}.
So indeed every cyclic-by-$p$ subgroup $H$ of $G$ is a (local) Oort group, proving part~(\ref{D2pn implies weak SOC}).

Part~(\ref{SOC special case}) then follows immediately, using that $D_{2p}$ is a local Oort group by \cite{BW}.
\end{proof}

Thus unconditionally, if $G$ is an $\Oo$-group whose order is odd or divisible by $p$ at most once (or at most twice if $p=2$), then $G$ is an Oort group.  

\begin{cor} \label{Oort O implications}
Let $k$ be an algebraically closed field of characteristic $p$.
Then the following assertions are equivalent:
\begin{enumerate}
\item \label{dihedral Oort}
For every $n$, the dihedral group $D_{2p^n}$ is an Oort group for $k$.
\item \label{SOC holds}
The Strong Oort Conjecture~\ref{strong Oort conj} holds.
\item \label{Oort O agree}
The class of Oort groups for $k$ is the same as the class of $\Oo$-groups for $p$. 
\end{enumerate}
\end{cor}

\begin{proof}
The implication (\ref{dihedral Oort}) $\Rightarrow$ (\ref{Oort O agree}) is immediate from 
Theorem~\ref{dihedral implies SOC}, and the implications 
(\ref{Oort O agree}) $\Rightarrow$ (\ref{SOC holds})
$\Rightarrow$ (\ref{dihedral Oort}) are trivial.
\end{proof}

We next turn to our classification results for $\Oo$-groups.  For $p$ odd, the almost simple $\Oo$-groups are given explicitly in Theorem~\ref{simple-odd} below.  Using that result, we will obtain the following classification theorem for more general $\Oo$-groups with $p$ odd (see the proof at the end of Section~\ref{odd char sect}):

\begin{thm} \label{main p odd}  Let $p$ be an odd prime.
Let  $G$ be a finite group.  Set
$R=O_{p'}(G)$.  Let $P$ be a Sylow $p$-subgroup of $G$.
Then $G$ is an $\Oo$-group if and only if:
$P$ is cyclic  and  one of
the following two (disjoint) conditions holds:
\begin{enumerate}
\item  \label{first odd case}
$G=RP$, or equivalently $N_G(P)=C_G(P)$;  or
\item  \label{second odd case}  $|N_G(P)/C_G(P)|=2$,
the  order $p$ subgroup $Q \le P$ has the property that
$C_G(Q)$ is abelian and every element of $N_G(Q) \smallsetminus C_G(Q)$
is an involution that inverts $C_G(Q)$.
\end{enumerate}
Moreover, if (\ref{second odd case}) holds,  then $R$ is solvable and either $G/R$ is a dihedral group of order $2p^a$
or $G/R$  is an almost simple group given in 
(\ref{almost simple odd case 1})-(\ref{norm equals cent case 5}) of Theorem~\ref{simple-odd}.
\end{thm}

If $G$ has odd order, the previous two results immediately give:

\begin{cor} \label{for:odd-odd}  Let $p$ be an odd prime and $G$ a group
of odd order.   Then the following are equivalent:
\begin{enumerate}
\item $G$ is an Oort group;
\item $G$ is an $\Oo$-group;
\item  a Sylow $p$-subgroup $P$ of $G$ is cyclic and $N_G(P)=C_G(P)$; and
\item  a Sylow $p$-subgroup $P$ of $G$ is cyclic and $G=RP$ where $R=O_{p'}(G)$.
\end{enumerate}
\end{cor}

For $p=2$,  there is an even more explicit description of $\Oo$-groups:

\begin{theorem}  \label{Oort 2 Sylow}
Let $G$ be a finite group with
a Sylow $2$-subgroup $P$.   Then $G$ is
an $\Oo$-group for $p=2$  if and only if either
\begin{enumerate}
\item \label{Oort 2 cyclic}
$P$ is cyclic, which implies $G = RP$ where we write $R$ for the solvable group $O(G)$; or
\item \label{Oort 2 dihedral}
$P$ is dihedral, and $C_G(K) = K$ for all elementary abelian subgroups $K \le P$ of order $4$.
\end{enumerate}
\end{theorem}

The above theorem is proven at the beginning of Section~\ref{char 2 sect}.
For the dihedral case of Theorem \ref{Oort 2 Sylow}, the next assertion provides
a more precise structure result.  Its proof appears at the end of Section~\ref{char 2 sect}.

\begin{thm}  \label{main2}
Let  $G$ be  an $\Oo$-group when $p = 2$.  Set
$R=O(G)$.  Let $P$ be a Sylow $2$-subgroup of $G$ and suppose $P$ is dihedral.
Set $N=N_G(P)$.
 Then
$[R,R]$ is nilpotent and
every elementary abelian subgroup of order $4$ in $G$ acts fixed point freely
on the non-identity elements of $R$.  Moreover, one of the following holds:
\begin{enumerate}
\item \label{no A4 main2}  $A_4$ is not a subgroup 
of $G$, and $G=RP$;  or
\item \label{A4 main2}  $A_4$ is a subgroup of $G$,  and
\begin{enumerate}
\item \label{A4 complement main2}  $G=RA_4$ is semi-direct and every $G$-chief factor of
$R$ is an irreducible $3$-dimensional module; or
\item \label{S4 complement main2}  $G=RS_4$ is semi-direct and every $G$-chief factor
of $R$ is an irreducible $3$-dimensional module in which
an element of order $4$ has trace $1$; or
\item \label{nonsolv main2} $G$ is not  solvable,  $R$ is nilpotent,
$G/R \cong \PSL(2,q)$ or $\PGL(2,q)$ with
$q>4$ a power of an odd prime $r$,
and all $G$ invariant sections of
$R$ are $3$-dimensional absolutely irreducible modules  (over
the $G$-endomorphism  ring).
Moreover,  $R$ is an $r$-group unless
possibly $q=5$ or $7$ with $G/R = \PSL(2,q)$.
\end{enumerate}
\end{enumerate}
\end{thm}

\section{$\Oo$-groups for odd primes} \label{odd char sect}

In this section we prove Theorem~\ref{main p odd}, classifying the finite groups that are $\Oo$-groups for an odd prime $p$.
We begin with two lemmas that apply for all primes $p$.

\begin{lemma}
\label{lem:sections} Let $G$ be an $\Oo$-group in characteristic $p$.
Then every section of $G$ is an $\Oo$-group.
\end{lemma}

\begin{proof} First note that it suffices to consider sections that
are quotient groups, since subgroups of $\Oo$-groups are
$\Oo$-groups.  We now reduce to the case in which $G$ is cyclic by $p$.
To make this reduction, it will suffice to show that if $\pi:G \to \Gamma$
is a surjection of groups and $T$ is a cyclic-by-$p$ subgroup of $\Gamma$,
there is a cyclic by $p$-subgroup $T'$ of $G$ such that $\pi(T') = T$.
Let $P_T$ be the (normal) Sylow $p$-subgroup of $T$.
Every $p$-Sylow subgroup $P_{T''}$ of $T'' = \pi^{-1}(P_T)$ surjects onto
$P_T$ under $\pi$, and all such $P_{T''}$  are contained in the kernel
$J$ of $\pi^{-1}(T) \to T/P_T$.  So if $t'$ is a element of $\pi^{-1}(T)$
whose image in $T/P_T$ generates $T/P_T$, we know that
$t' P_{T''} t'^{-1} = t_0 P_{T''} t_0^{-1} \le J$ for some $t_0 \in J$.
Now $t = t_0^{-1} t'$ normalizes $P_{T''}$ so we can take $T'$
to be the group generated by $t$ and $P_{T''}$.

If $p$ is odd,  this implies that either $G$
is cyclic (and so every section is also cyclic)  or
$G$ is dihedral of order $2p^a$  (and so every section
is either cyclic or dihedral of order $2p^b$).

If $p=2$,  $G$ is either cyclic,  $A_4$ or a dihedral
$2$-group and the  result is clear.
\end{proof}

We will use the above  result often and usually without comment.

See \cite{italy} for the next result.  All proofs of this lemma
seem to  require the classification of finite simple groups.
Here $C_G(\sigma)$ denotes the subgroup of $G$ consisting of
elements fixed by an automorphism $\sigma$.

\begin{lemma}  \label{solv if abel cent}
Let $G$ be a finite group with $\sigma$ an
automorphism of order coprime to $|G|$.  If
$C_G(\sigma)$ is abelian,  then $G$ is solvable.
\end{lemma}

{\em For the remainder of the section, we fix an odd prime $p$.}
Let $G$ be a finite group with
a Sylow $p$-subgroup $P$, and let 
$\Cp(G)$ be the set of cyclic-by-$p$ subgroups
of $G$.
Our goal is to characterize all $\Oo$-groups with respect to $p$; i.e.\ all
finite groups $G$
such that every $H \in \Cp(G)$ either is cyclic or
is  dihedral of order twice a power of $p$.

In \cite[\S3]{cgh1}, the conclusion of the following result was proven under the (a priori stronger) hypothesis that $G$ is an Oort group:

\begin{lemma} \label{two cases odd}
Let $G$ be an $\Oo$-group.  Then:
\begin{enumerate}
\item \label{odd Sylow cyclic} A Sylow $p$-subgroup $P$ of $G$ is cyclic. 
\item \label{odd Sylow cases} If $1 \ne Q \le P$,  then either
 \begin{enumerate}
 \item \label{case norm equals cent} $N_G(Q)=C_G(Q)$, or
 \item \label{case norm unequal cent} $N_G(Q)/C_G(Q)$ has order $2$ and every element in
$N_G(Q)$ either centralizes $Q$ or is an involution that acts as
inversion on $C_G(Q)$, and $C_G(Q)$ is abelian.
 \end{enumerate}
\end{enumerate}
\end{lemma}

\begin{proof}  Part (\ref{odd Sylow cyclic}) is a trivial consequence of the
hypothesis.   For part~(\ref{odd Sylow cases}), note that $Q$ is a cyclic $p$-group by part~(\ref{odd Sylow cyclic}).
If $x \in N_G(Q)\smallsetminus C_G(Q)$, the group $\langle x,Q\rangle$ is a cyclic-by-$p$
group which is not abelian.  Hence since $G$ is an $\Oo$-group, $\langle x, Q\rangle$ is
a dihedral group of order $2p^n$ for some $n$.  Since $x$ does not centralize $Q$
this means $x$ must be an involution.
Choose one such involution $x$.  If $g \in C_G(Q)$,  then
$xg$ is also in $N_G(Q)\smallsetminus C_G(Q)$, so it is also an involution.  This implies $x$ acts as inversion
on $C_G(Q)$.  Since conjugation by $x$ is a homomorphism, $C_G(Q)$ is abelian.
\end{proof}

It is easy to handle one case.

\begin{lemma}  \label{easy odd case}
Let $G$ be a finite group with a cyclic Sylow $p$-subgroup $P$.
Let $R = O_{p'}(G)$. The following are equivalent:
\begin{enumerate}
\item \label{odd semidirect}  $G = RP$.
\item \label{odd norm equals cent for Sylow} $N_G( P ) = C_G ( P )$.
\item \label{odd norm equals cent for all} $N_G( Q ) = C_G ( Q )$ for every non-trivial subgroup $Q$ of $P$.
\end{enumerate}
If any of these conditions hold then $G$ is an $\Oo$-group.
\end{lemma}

\begin{proof}  Suppose first that $N_G( P ) = C_G ( P )$.  By Burnside's normal $p$-complement
theorem \cite[p.~419]{huppert},
$G=R'P$ with $R'$ a normal $p'$-subgroup which is a complement to $P$.  Because $P$ is a $p$-group, $R'$ must
be $R= O_{p'}(G)$ so $G = RP$.

Suppose now that $G = RP$.  Let us prove that $N_G( P ) = C_G ( P )
$.
  We have
$N_G( P ) = R' P$ where $R' = N_G ( P ) \cap R$.  Now both $P$ and $R'$ are normal in
$N_G( P )$ and they have trivial intersection, so $N_G( P ) $ is isomorphic to the product
of $P$ and $R'$. This implies $N_G ( P ) = C_G( P )$.

We have now shown that conditions (\ref{odd semidirect}) and (\ref{odd norm equals cent for Sylow}) are equivalent, and clearly (\ref{odd norm equals cent for all}) implies
(\ref{odd norm equals cent for Sylow}). So it will suffice to show that if (\ref{odd semidirect}) and (\ref{odd norm equals cent for Sylow}) hold then (\ref{odd norm equals cent for all}) holds.
Let $Q$ be a subgroup of $P$
and let $G' = R Q$.  
Then $N_{G'}(Q) = C_{G'}(Q)$ by applying the equivalence of conditions
(\ref{odd semidirect}) and (\ref{odd norm equals cent for Sylow}) for the group $G'$.  So since $P$ is cyclic, we get
$N_G(Q) = \langle P, N_{G'}(Q)\rangle = \langle P, C_{G'}(Q)\rangle = C_G(Q)$.

For the final assertion in the Lemma, we suppose condition (\ref{odd norm equals cent for all}) holds.  If $Q$ is any subgroup
of $P$,  then $N_G(Q) =C_G(Q)$. Since $P$ is cyclic, this implies that every cyclic-by-$p$ subgroup of $G$
is in fact cyclic, so  $G$ is an $\Oo$-group.
\end{proof}

We now turn to the other case of Lemma~\ref{two cases odd},
i.e.\ of $\Oo$-groups $G$ such that
$N_G(P)  \ne C_G( P )$.

\begin{lemma} \label{basic1}
Let $G$ be an $\Oo$-group with Sylow $p$-subgroup $P$, such that $N_G(P)  \ne C_G(P)$.
Let $\tau$ be an element of $N_G(P) \smallsetminus C_G(P)$.  Then $\tau$ is an
involution  that acts as inversion
on the abelian subgroup $C_G(P)$.   Since $P$ is cyclic,
$\tau$ normalizes each subgroup of $P$.
Let $1 \ne Q \le P$.
\begin{enumerate}
\item  Every element
of $\tau C_G(Q)$ is an involution and acts as inversion
on $C_G(Q)$.
\item    $C_G(Q)=C_G(P)$ is abelian.
\item    $N_G(Q)=N_G(P)$.
\item    Every member of the set $\Cp(G)$ of cyclic by $p$ subgroups of $G$ either has order
prime to $p$ or is conjugate
to a subgroup of $N_G(P)$.
\end{enumerate}
\end{lemma}

\begin{proof} The first statement concerning $\tau$
was shown in  Lemma~\ref{two cases odd}.
If $g \in C_G(Q)$,  then $\langle Q,  \tau g \rangle$
is in $\Cp(G)$ and is not cyclic,  whence it must be dihedral
of order $2p^a$ for some $a$.  In particular,  every element
either is a $p$-element or has order $2$.  Thus,
$\tau g$ is an involution and so $\tau$ inverts $g$.
This proves the first statement.

Since inversion is an automorphism of $C_G(Q)$,  it follows
that $C_G(Q)$ is abelian.  Since $C_G(Q)$ contains $P$ we conclude that
$C_G(Q)=C_G(P)$.   By Lemma \ref{two cases odd},  $N_G(Q)  = \langle C_G(Q),  \tau \rangle
=  \langle C_G(P),  \tau \rangle = N_G(P)$.

If $X \in \Cp(G)$ and does  not have order prime to $p$,
then,  by conjugating,  we may assume that $Q:=O_p(X)  \le P$,
whence $X \le N_G(Q)=N_G(P)$.
\end{proof}

This gives the following nice criterion for checking for
$\Oo$-groups.

\begin{corollary}
\label{cor:OddCriterion} Let $p$ be an odd prime,  $G$ a finite
group with cyclic Sylow $p$-subgroup $P$.  The following conditions are equivalent:
\begin{enumerate}
\item \label{cond:Ogroup} $G$ is an $\Oo$-group.
\item \label{cond:NOmegaOgroup} $N_G(\Omega_1(P))$ is an $\Oo$-group.
\item \label{cond:involution} Either $N_G(P)=C_G( P )$, or there is an involution
$\tau$ inverting the abelian group $C_G(\Omega_1(P))$
and $N_G(\Omega_1(P))  = \langle C_G(\Omega_1(P)),  \tau \rangle$.
\end{enumerate}
\end{corollary}

\begin{proof} That (\ref{cond:Ogroup}) implies (\ref{cond:NOmegaOgroup}) is immediate.  Since the $p$-Sylow subgroups
of $p$ are cyclic, any cyclic-by-$p$ subgroup of $G$ can be conjugated into
$N_G(\Omega_1(P))$, so (\ref{cond:NOmegaOgroup}) implies (\ref{cond:Ogroup}).  To show (\ref{cond:NOmegaOgroup}) implies (\ref{cond:involution}), note
first that the normalizer in $G$ of any non-trivial subgroup of $P$ is contained
in $N_G(\Omega_1(P))$.  Hence (\ref{cond:NOmegaOgroup}) implies (\ref{cond:involution}) on replacing $G$ by $N_G(\Omega_1(P))$
and on letting $Q = \Omega_1( P ) $ in Lemmas  \ref{two cases odd} and \ref{basic1}.
Suppose finally that condition (\ref{cond:involution}) holds.  If $N_G( P ) = C_G ( P )$ then (\ref{cond:Ogroup}) holds
by Lemma \ref{easy odd case}. So we now suppose that there is an involution
$\tau$ as in part (\ref{cond:involution}).  Suppose that $X$ is a cyclic by $p$ subgroup of $G$.
By conjugating $X$ inside $G$, we can suppose that $X$ is contained in $N_G(\Omega_1 ( P ))$, which by hypothesis
is $\langle C_G(\Omega_1( P )), \tau)$.  From the fact that $\tau$ is an involution
inverting the abelian group
$C_G(\Omega_1 ( P ))$, we now see that $X$ is either cyclic or dihedral of order $2 p^n$
for some integer $n$.  Hence (\ref{cond:Ogroup}) holds.
\end{proof}

Note that the above results prove that $G$ is an $\Oo$-group if and only if its Sylow $p$-subgroup is cyclic and either (\ref{first odd case}) or (\ref{second odd case}) of Theorem \ref{main p odd} holds; see the proof of Theorem \ref{main p odd} at the end of this section for details.  The following results will enable us to complete the proof of that theorem.

\begin{lemma}
\label{lem:solv}
Let $G$ be an $\Oo$-group with Sylow $p$-subgroup $P$, such that $N_G(P)  \ne C_G(P)$.
Then
$R:=O_{p'}(G)$ is  solvable.
\end{lemma}

\begin{proof} The hypotheses imply that the $p$-Sylow $P$ is cyclic.  Let $Q$ be the
subgroup of $P$ of order $p$.  We know furthermore by Lemma \ref{two cases odd} that there is an involution $\tau \in G$
which normalizes $Q$.  The group $D$ generated by $Q$ and $\tau$ is dihedral
of order $2p$.  If $\tau$ were in the normal $p'$-subgroup $R$, then we would
have $g \tau g^{-1} \in R$ when $g$ generates $Q$.  Then  $\tau^{-1} g\tau g^{-1} = g^{-2} \in R$, but this is a non-trivial element of $Q$ because $p$ is odd, contradicting the fact that $R$
has order prime to $p$. Thus $\tau$ does not lie in $R$, so the semi-direct product
$RD$ has order $|R| \cdot 2p$.  We now observe that $R' = O_{p'}(RD)$
equals $R$, since $R'$ contains $R$ and $R'/R$ is a normal $p'$-subgroup of the dihedral group $D = RD/R$ of order $2p$.  Hence we can reduce to the case in which
 $G = RD$, so now $P = Q$.

Passing to a quotient,  we may also assume that $R$ contains no non-trivial normal
solvable subgroup of $G$.  We will suppose $R$ is non-trivial and obtain a contradiction.
Since $R$ is normal in $G$, there is a non-trivial minimal normal subgroup $N$ of $G$
which is contained in $R$. This $N$ must be the product of some number of
copies of isomorphic simple groups, which by assumption must be non-abelian.
  Thus,  $P$ acts on $R$ and on $N$, and $C_R(P)$ is abelian because $C_G( P ) $
is abelian by Lemma \ref{basic1}.

Suppose there is a simple factor of $N$ which is not fixed by $P = Q$.  Then $P$ would permute some
number of factors. Because $|P| = p$, $P$ would then centralize a subgroup of $N$ which is a diagonally embedded copy of one simple factor.
This could contradict the fact that $C_R( P ) $ abelian since the simple factors
of $N$ are not abelian.  Thus $P$ fixes each factor of $N$.

 Let
$N_0$ be one of the simple factors of $N$, so that $P$ normalizes
$N_0$.  A generator $\sigma$ of $P$ then acts on $N_0$.  The
centralizer $C_{N_0}(\sigma)$ is contained in the abelian
group $C_R( P )$ so it is abelian. Therefore by
Lemma \ref{solv if abel cent}, $N_0$ must be solvable, contradicting
the fact that it is a non-abelian simple group. This  contradiction completes the proof.
\end{proof}

We next classify the almost simple $\Oo$-groups.
This seems to  require the
classification of finite simple groups.

\begin{theorem} \label{simple-odd}  Let $G$ be a finite group and let $p$ be an odd prime.  Suppose that $G$ is almost simple, i.e.\ that $S:=F^*(G)$ is a non-abelian simple group.
Then $G$ is an $\Oo$-group if and only if
one of the following holds:
\begin{enumerate}
\item \label{almost simple odd case 1}
$G=\PSL(2,q)$ or $\PGL(2,q)$ with $q > 3$ and $p$ dividing $q^2-1$.
\item  \label{almost simple odd case 2}
$G= S = \Sz(q)$ with $q=2^{2k+1} > 2$ and $p|(q-1)$.
\item  \label{almost simple odd case 3}
$G= S = \Ree(q)$ with $q=3^{2k+1} > 3$ and $p$ dividing $q-1$.
\item  \label{almost simple odd case 4}
$S =\PSL(3,4)$ and $[G:S]=1,2$ or $4$ with $p=5$.
\item \label{norm equals cent case 5} $N_G( P ) = C_G( P )$ for $P$ a $p$-Sylow subgroup of $G$.
\end{enumerate}
\end{theorem}

\begin{proof}   We begin by supposing that $G$ is an $\Oo$-group. We can also suppose that
(\ref{norm equals cent case 5}) does not hold.  Therefore $p$ divides the order of $G$, and by Lemma \ref{two cases odd},
a $p$-Sylow $P$ of $G$ is cyclic and $N_G(P)/C_G( P ) $ has order $2$, where
there is an involution $\tau \in G$ which inverts $P$.
 Then $S$ is also an $\Oo$-group.

 Let us show that $S$ contains $P$.  Let $H = SP$. Then $H$ is an $\Oo$-group since $H \subset G$.
 If $C_H(P) = N_H( P )$  then by Burnside's normal $p$-complement theorem, $H = O_{p'}(H) P$.
 Since $P$ is cyclic, and $S$ a  non-abelian simple group, the image of $S$ in $H/O_{p'}(H)$
 is trivial, so $S \subset O_{p'}(H)$.  But $S = F^*(G)$ is normal in $G$, so $H = SP = O_{p'}(H)P$
 implies $S = O_{p'}(H)$.  By Lemma \ref{lem:solv}, $O_{p'}(G)$ is solvable.
 But $S$ is normal in $G$ and $S = O_{p'}(H)$ has order prime to $p$, so $S \subset O_{p'}(G)$.
 This is a contradiction because $S$ is not solvable.   Therefore, $C_H(P) \ne N_H( P )$
 so $N_H( P ) $ is generated by $C_H( P )$ and an involution $\tau \in N_H( P )$
 which inverts $P$ by Corollary \ref{cor:OddCriterion}.  Since $p$ is odd, this means $P \subset [H,H]$, and $[H,H] = S$
 since $H = SP$.  Thus $P \subset S = H$ and $N_S( P ) \ne C_S( P )$.  In particular,
 case (\ref{norm equals cent case 5}) does not hold when $G$ is replaced by $S$;  we will need this fact later.

We now first
consider the case that $S=G$.
In particular,  Lemmas \ref{two cases odd} and \ref{basic1}  shows that if
$1 \ne x \in P$,  then $|x^S \cap P|  = 2$,
where $x^S$ is the set of $S$-conjugates of $x$.
If $S$ is a sporadic group,  this can never happen  (see
\cite[5.3]{gls3}).  If $S$ is the alternating group $A_n$,  then an element
of order $p$ has at least $(p-1)/2$ conjugates,  whence $p \le 5$.
Since $P$ is cyclic,  this forces $n=5$ for $p=3$.
Then $S=A_5 \cong \PSL(2,4)$.   If $p=5$ and $n \ge 7$,  then
a $p$-cycle is  rational  (i.e.\ is conjugate to all its non-trivial
powers).  So $n=5$ or $6$.  Thus,  $S=G=\PSL(2,q)$ with $q=4$ or $9$.
So if $S = A_n$ we are in case (\ref{almost simple odd case 1}).

So we may assume that $S$ is a simple finite group of Lie type
in characteristic $r$.  If $p = r$,  then $S=\PSL(2,p)$  (since
the Sylow $p$-subgroup is cyclic and so in particular there is a most
one  root subgroup).  An element of order $p$ is conjugate to $(p-1)/2$
of its powers,  whence $p \le 5$.  Since $S$ is simple,  $p \ge 5$ and
so $p=5$ is the only possibility.  Then $S=G=\PSL(2,5)  \cong \PSL(2,4)$
and we are in case (\ref{almost simple odd case 1}).

So assume that $p \ne  r$.
We require some facts about maximal tori in the finite groups
of Lie type.  See  \cite[\S 14]{St311} and \cite[II.1]{ss} for details.  Since $P$ is cyclic and its centralizer
is abelian,  it follows that $C_S(P)$ is a maximal torus $T$ and
$P$ is generated by a semi-simple  regular element.   In particular,
$T$ is a non-degenerate maximal torus.  These are parametrized
by  (twisted)  conjugacy classes of elements in the Weyl group $W$
of $S$.  Moreover,  $N_G(T)/T$ is isomorphic to the centralizer
of this element in the Weyl group and has order $2$.
Inspection of the Weyl groups  shows that this forces
$S$ to be either a  rank one Chevalley group  (i.e.  $S \cong \PSL(2,q)$,
$\PSU(3,q)$,  $\Sz(q)$ or $\Ree(q)$)  or else $S = \PSL(3,q)$.

In the case that $S = \PSL(3,q)$,  $T$ must have an irreducible $2$-dimensional
subspace in the natural $3$-dimensional module.
There is an element $s \in N_S( P )$ which acts Frobenius on
$T$.  Hence $C_T(s)$ is cyclic of order  $(q-1)/(3,q-1)$.  But $s$ also
acts by inversion on $T$, by Lemma \ref{basic1}.  Hence  $(q-1)/(3,q-1) \le 2$,
whence $q=2,4$ or $7$.   Moreover,  $p|(q+1)$.  We can't have $q = 7$
since $p$ is odd.  If $q = 2$, then $S = \PSL(3,2) = \PSL(2,7)$, as
in case (\ref{almost simple odd case 1}). Finally, if $q=4$ then $p=5$ as
in case  (\ref{almost simple odd case 4}).

Now consider the  rank $1$ Chevalley groups.  The only maximal torus  in
the Suzuki or Ree groups with automizer of order $2$ is the
quasi-split torus  (i.e.  the torus contained in a Borel subgroup
and so the torus has order $q-1$;  note that this  rules
out the group $\Ree(3)'$).  Note that since $P \subset T$ this
means $p|(q-1)$.  The other conditions on $p$ and $q$
for the Suzuki and Ree groups follow from the fact that $S$
is simple, so these cases fall into Cases (\ref{almost simple odd case 2})
and (\ref{almost simple odd case 3}).  For the group $\PSL(2,q)$, we must have
$q > 3$ and $p|(q^2 - 1)$ as in case (\ref{almost simple odd case 3}).
If $S=\PSU(3,q)$,  we argue as above in the case of $\PSL(3,q)$
to see that $q =5$ and $p|(q+1)$.  However,  the Sylow $3$-subgroup
is not cyclic,  a contradiction.

We now continue to suppose $G$ is an $\Oo$-group, but we drop the
assumption that $G = S$.  The above arguments show that $S$ must
be as (\ref{almost simple odd case 1}) - (\ref{almost simple odd case 4}), since $S$ is an $\Oo$-group and we proved earlier that (\ref{norm equals cent case 5}) cannot hold for $S$.
Since $P \subset S$ is a $p$-Sylow of $S$,  Sylow's Theorems and the Frattini argument
imply that $G=SN_G(P)$.

Since $S$ is not in case (\ref{norm equals cent case 5}),
$|N_S(P)/C_S( P )| = 2$ by Lemma \ref{two cases odd} and
there is an involution $\tau$ of $N_S(P)$ which acts by inversion on $P$.
If $\beta$ is in $N_G( P )$ then either $\beta \in C_G( P ) $ or $\tau \beta \in C_G( P )$.
So since $\tau \in S$ and $G = S N_G(P)$ we conclude that  $G=SC_G( P ) $.
Since $\tau \in S$ and $S$ is normal in $G$, we have we have
$[\tau ,C_G( P )]  \le S$.  Because $\tau$ inverts $C_G(P)$, we have $\tau c \tau^{-1} c^{-1} = c^{-2}$
for $c \in C_G( P ) $ so all  squares in  $C_G(P)$ are
contained in $S$.  Thus, $G/S$ is an elementary abelian $2$-group.

Let us prove that $G =S$ in cases  (\ref{almost simple odd case 2})  and  (\ref{almost simple odd case 3}).    Since $S = F^*(G)$ is a non-abelian simple group, we know  
$S \subset G \subset \mathrm{Aut}(S)$.
However, for the $S$ in (\ref{almost simple odd case 2}) and (\ref{almost simple odd case 3}), 
$\mathrm{Aut}(S)/S = \mathrm{Out}(S)$ has odd order,
so the $2$-group $G/S$ is trivial and $G = S$ and we have already treated this case.

Suppose $S  = \PSL(2,q)$ as in case  (\ref{almost simple odd case 1}).  We have shown
$G = S C_G(P)$ and that $S\subset G \subset \mathrm{Aut}(S)$, with $G/S$ an elementary
abelian $2$-group.  Further $P \subset S$ is cyclic of order $p^n$ for some odd prime $p$,
and $p$ is prime to $q$. Therefore a generator $g$ of $P$ lifts to an element $\tilde{g} \in \SL(2,q)$
of order $p^n$, and $\tilde{g}$ is semi-simple with distinct eigenvalues.  Thus the sub-algebra $A$
of $\mathrm{Mat}_2(q)$ generated over $\mathbb{F}_q$ by $\tilde{g}$ is \'etale of dimension $2$,
and a maximal torus $T$ containing $P$ is the elements of $A$ of norm $1$ to $\mathbb{F}_q$.
Thus $C_G(P) = C_G(T)$.  Only an outer automorphism of  $S = \PSL(2,q)$
coming from conjugation by an element of $\PGL(2,q)$ centralizes
$T$. Since $G = S C_G( P ) = S C_G(T)$,
$G=\PGL(2,q)$ or $G = \PSL(2,q)$ are the only possibilities.

Suppose now that $S = \PSL(3,4)$ is as in  case  (\ref{almost simple odd case 4}).  We know $\aut(S)/S$
is the dihedral group of order $12$, and $S \subset G \subset \aut(S)$.  Since $G = S C_G( P)$,
we find that $|G/S| =1,2$ or $4$.  This completes the proof of `only if' part of Theorem \ref{simple-odd}.

As for the `if' part of the proof,  it is  straightforward to see that if $G$ is
any of the groups listed in (\ref{almost simple odd case 1}) - (\ref{almost simple odd case 3}) of the theorem, then  the normalizer of any  non-trivial
$p$-subgroup is dihedral. Therefore $G$ is an $\Oo$-group.  With the above notation we
showed groups for which (\ref{norm equals cent case 5}) is true are $\Oo$-groups by Lemma \ref{easy odd case}.

Suppose now
that $G$ is as in case (\ref{almost simple odd case 4}), so $S =\PSL(3,4)$ and $[G:S]=1,2$ or $4$ with $p=5$.
Then $S \subset G \subset \aut(S)$ and  $\aut(S)/S$
is the dihedral group of order $12$. Thus, up to conjugation, we may assume
that $G \le  H:=\langle S, x, y \rangle$ where $x$ induces transpose inverse on $S$ and
$y$ induces the Frobenius automorphism of order $2$ on $S$.
Every Sylow $5$-subgroup
of $H$ is of order $5$ and contained in $S$.  By counting elements of order $5$ in $S$
and using the Sylow theorems,
one finds that $N_S(P)$ is dihedral of order $10$, $N_H(P)$ has order $40$, and $N_H(P)$
surjects onto the Klein four group $H/S$.  To show $H$ is an $\Oo$-group,
it will suffice to show every element of $N_H(P)$ not in  $P$ has order $2$.
Each such element has the form $sz$ for some $z \in \{e,x,y,yx\}$ and $s \in S$.
If $z = e$ then $s$ has order $2$ since $N_S(P)$ is dihedral of order $10$.
Suppose now that $z \ne e$, so that $(sz)^2 = szsz = s s'$ where $s' = z s z =z s z^{-1}$
is the image of $s$ under the involution of $S$ corresponding to $z$.
We can realize $P$ as a subgroup
of $\PSL(2,4) = \SL(2,4)$ inside $\PSL(3,4)$ via the the embedding of $\SL(2,4)$
into $\SL(3,4)$.  We can furthermore arrange that $xPx^{-1} = P$ while
$yPy^{-1}$ is a different $5$-Sylow subgroup of $\PSL(2,4)$ embedded into $\PSL(3,4)$.
Since all $5$-Sylow subgroups of $\SL(2,4)$ are conjugate within $\SL(2,4)$,
we conclude that any $s$ as above must also lie in $\SL(2,4)$ inside $\PSL(3,4)$.
One then calculates using matrices that in fact $(sz)^2 = s s' = e$ for all such $s$, which completes
the proof of the `if' direction of Theorem \ref{simple-odd}.
\end{proof}

\begin{theorem} \label{main odd second case}
Let $G$ be an $\Oo$-group with Sylow $p$-subgroup $P$, such that $N_G(P)  \ne C_G(P)$.
Let $R=O_{p'}(G)$.
Then either
$G=RD$ with $D$ dihedral of order $2p^a$
for some $a > 0$ or else $G/R \cong S$,  where $S$ is an almost simple
group given
in the previous theorem.
\end{theorem}

\begin{proof}
First suppose that $G/R$ has a non-trivial normal $p$-subgroup.  Let
$Q$ be the subgroup of order $p$ in $P$. Then $RQ/R$ is a subgroup
of order $p$ in a (cyclic, normal, non-trivial) $p$-subgroup of $G/R$ so $RQ/R$ is normal
in $G/R$. Now $RQ$ is normal in $G$, so the Frattini argument shows
$G = R N_G(Q)$.  Since $N_G(Q)$ contains $C_G(Q) = C_G(P)$ with index $2$,
this means $R C_G(P)$ is normal of index $1$ or $2$ in $G$.  We have shown
$C_G(Q) = C_G(P) = A \times P$ where $A$ is an abelian subgroup of order prime to $p$
and the involution $\tau$ of $N_G(Q)$ acts by inversion on $C_G(Q)$.  This implies
that $RA$ is stable under conjugation by $P$ and by $\tau$, so $RA$ is in fact
normal in $G$. Thus $A \subset R = O_{p'}(G)$.  Let $D$ be the dihedral group
of order $2|P|$ generated by $\tau$ and $P$.  We conclude that $G = RD$
is the semi-direct product of $R$ and $D$, since $D \cap R$ must be a normal
$p'$-subgroup of $D$, which forces $D \cap R$ to be trivial.

So we may assume that $G/R$ has no non-trivial normal $p$-subgroup.
By Lemma \ref{lem:sections},
we may pass to $G/R$ and assume that $R=1$.
Thus every normal subgroup of $G$ has order divisible
by $p$ but is not a $p$-group.

In other words,  $O_p(G)=
O_{p'}(G)=1$.  In particular,  $F(G)=1$ and so $F^*(G) = E(G)$ is a direct product
of subnormal quasi-simple subgroups.  Here $O_p(G) = O_{p'}(G) = 1$ implies
$Z(G) = 1$, so all the factors in $E(G)$ are simple and non-abelian of order divisible by $p$.
Since the Sylow $p$-subgroup of $G$ is cyclic,  this implies
that $S:=F^*(G) = E(G)$ is simple.  Thus $S \le G \le \aut(S)$.
The  result now follows by Theorem \ref{simple-odd}.
\end{proof}

Using the above results, we can now prove Theorem \ref{main p odd}:

\begin{proof}[Proof of Theorem \ref{main p odd}]
Suppose first that $G$ is an $\Oo$-group.  By Lemma~\ref{two cases odd}, a Sylow $p$-subgoup $P$ of $G$ is cyclic.
Then $N_G(P) = C_G(P)$ if and only if $G = RP$ by Lemma~\ref{easy odd case}.  Suppose
now that $N_G(P)  \ne C_G(P)$. Then condition~(\ref{case norm unequal cent}) 
of Lemma~\ref{two cases odd} holds for
the group $Q = P$.  This means that there is an involution $\tau$ which acts
by inversion on the abelian group $C_G(P)$.  Here $P \subset C_G(P)$ and
$P$ is cyclic of odd order, so $\tau$ acts non-trivially by inversion on every
non-trivial subgroup $Q$ of $P$.  This means that condition~(\ref{case norm unequal cent}) of Lemma~\ref{two cases odd} holds for all such $Q$.  We have now shown that if $G$ is an $\Oo$-group
then condition (\ref{first odd case}) or (\ref{second odd case}) of Theorem~\ref{main p odd} must hold.
Conversely, if a $p$-Sylow subgroup $P$ of $G$ is cyclic and one of
conditions (\ref{first odd case}) or (\ref{second odd case}) of Theorem~\ref{main p odd} holds,
then $G$ is an $\Oo$-group by Corollary \ref{cor:OddCriterion}.
The last statement of Theorem \ref{main p odd} follows immediately
from  Lemma \ref{lem:solv}  and Theorem \ref{main odd second case}.
\end{proof}

\smallskip

In general,  it seems difficult to give a more detailed classification of $\Oo$-groups for $p$ odd, i.e.\ to say much more than that $D$ acts on
the solvable group $R$ with $C_R(P)$ abelian and inverted by $D$
(with notation as in Theorem~\ref{main odd second case}).

\begin{problem}  Classify the groups $R$ admitting such an action
by a dihedral group of order $2p^a$.
\end{problem}

\section{$\Oo$-groups for the prime $2$} \label{char 2 sect}

In this
section,  we characterize $\Oo$-groups in characteristic
$2$  ---  i.e.\ those finite groups $G$ such that every cyclic-by-$2$ subgroup
$H$ of $G$ is either cyclic,  a dihedral
$2$-group or isomorphic to $A_4$.  We begin by proving Theorem~\ref{Oort 2 Sylow} and
conclude by proving Theorem~\ref{main2}.

\begin{proof}[Proof of Theorem \ref{Oort 2 Sylow}]
If $G$ is an $\Oo$-group,  then $P$ must be either cyclic or dihedral.
The group $R = O(G) = O_{2'}(G)$ is solvable by the Feit-Thompson Theorem.   First suppose that
$P$ is cyclic. Burnside's normal $p$-complement theorem
says that $R$ is a normal complement to $P$ in $G$ if $P$ is central in its normalizer $N_G(P)$.
If $\sigma \in N_G(P)$ has order prime to $p$, then conjugation by $\sigma$ must induce
the trivial automorphism of $P$ since $\mathrm{Aut}(P)$ is a $2$-group.  Thus
$\sigma$ centralizes $P$, and  $G = RP$  follows.

Now suppose $P$ is dihedral.  Then $P$ contains a Klein four subgroup $K$.
For each such $K$, there can be no subgroup of $G$ of the form $K \times X$ with $X$ non-trivial, 
since $G$ is an $\Oo$-group.
Since $C_P(K)=K$,
it  follows that $C_G(K)=K$.

So now assume conversely that either $P$ is cyclic, or else $P$ is dihedral
and that each elementary abelian subgroup $K \le P$ of order $4$ satisfies $C_G(K) = K$.

Let  $H \le G$ be a cyclic-by-$2$ subgroup; i.e.\ $H/O_2(H)$ is cyclic of odd order.  We may assume that
$O_2(H)  \le P$.   Then $O_2(H)$ is either
cyclic or dihedral.  In order to show that $G$ is an $\Oo$-group, we want to show that $H$ is either cyclic, or a dihedral $2$-group, or is isomorphic to $A_4$.

Suppose first that $O_2(H)$ is not an elementary abelian subgroup
of order $4$.  Then the automorphism group of $O_2(H)$ is a $2$-group,
whence $H = O_2(H)  \times O(H)$ with $O(H)$ cyclic.  If $O_2(H)$
is cyclic,  then $H$ is cyclic.  If $O_2(H)\le P$ is not cyclic,  then
$O_2(H)$ contains a Klein four subgroup $K$.  But then
$O(H)$ centralizes $K$, whence $O(H)=1$ because we have assumed
$C_G(K)= K$.  Therefore $H = O_2(H)$ is a cyclic or dihedral
$2$-group.

Finally, suppose $K:=O_2(H)$ is elementary abelian
of order $4$.  Then $H$ is contained in the normalizer $N_G(K)$, and $N_G(K)/K = N_G(K)/C_G(K)$ embeds in $\aut(K)  \cong
S_3$.  This implies  $H=K$ or $A_4$ since $K = O_2(H)$ is the $2$-Sylow subgroup of $H$.  So again $H$ is of the permissible form, concluding the proof.
\end{proof}

\begin{remark} 
\ (\ref{Oort 2 cyclic}) \  
In Theorem~\ref{Oort 2 Sylow}(\ref{Oort 2 cyclic}), 
there is also a more elementary proof that $G = RP$ if $P$ is cyclic when $p = 2$.
Namely, embed $G$ in $S_n$  via the  regular representation, and
observe that $G$ does not embed in $A_n$.  So $G$ contains
a normal subgroup of index $2$.  The assertion follows by induction.

(\ref{Oort 2 dihedral}) \  
In Theorem~\ref{Oort 2 Sylow}(\ref{Oort 2 dihedral}),
one must have
the condition for {\em all} elementary abelian Klein $4$-subgroups,
not just a single one.  This can be seen by
considering the semi-direct product $G=AD_8$, where $A$ is abelian of odd order and
$D_8$ acts by inversion on $A$ with the kernel of the action being
an elementary abelian subgroup $K$ of order $4$.  Here 
the other Klein four subgroup of $D_8$ is self-centralizing in $G$, but $K$ is not, 
and $G$ is not an $\Oo$-group.
\end{remark}

In order to prove Theorem~\ref{main2}, we now further study the structure of $\Oo$-groups in characteristic $2$.

If the $2$-Sylow subgroup $P$ of an Oort group $G$
is not cyclic, then $P$ is dihedral, as in Theorem~\ref{Oort 2 Sylow}.
In this situation,
we next investigate the center of $G$.

\begin{lemma} \label{nontrivial center lemma}
Let $G$ be an
$\Oo$-group in characteristic $2$ with a non-cyclic Sylow $2$-subgroup $P$,
and set $R=O(G)$.
Assume that $Z(G)  \ne 1$.  Then $|Z(G)|=2$ or $4$,
and $G$ is solvable.  Moreover
one of the following holds:
\begin{enumerate}
\item \label{nontrivial center 1}
$G$ is elementary abelian of order $4$.
\item \label{nontrivial center 2}
$|Z(G)| =2$,  $R$ is abelian and $G=RP$.
Moreover $C_P(R)$ is a cyclic subgroup of index $2$ in $P$,
and all elements of $P$ not in $C_P(R)$ induce inversion on $R$.
\end{enumerate}
\end{lemma}

\begin{proof} Let $K$ be an elementary abelian subgroup of order $4$
contained in $P$.   Since $C_G(K)=K$ by Theorem~\ref{Oort 2 Sylow},
we see that $Z(G)  \le K$.  This proves the first statement, on the order of $Z(G)$.
Note that $|Z(P)|=2$ unless $P=K$.  So if $|Z(G)|=4$,  then
$G=C_G(K) = K$ and so (\ref{nontrivial center 1}) holds.

The remaining case is that
$Z(G)$ has order $2$.
This condition implies that $Z(G)$ is contained in every elementary
abelian subgroup of order $4$ of $G$.

We claim that this implies that $N_G(X)/C_G(X)$ is a $2$-group for $X$
any non-trivial $2$-subgroup of $G$.  If $X$ is cyclic or dihedral
of order greater than $4$,  this is clear since the automorphism
group of $X$ is a $2$-group.   If $X$ is elementary  abelian of
order $4$,  then $X \cap Z(G)  \ne 1$.  The unique non-trivial
element $\sigma$ of $X \cap Z(G)$ must be fixed by conjugation by every
element of $N_G(X)$. We conclude that $N_G(X)/C_G(X)$ embeds
into the group of automorphisms of the Klein four group $X$ which fix $\sigma$.
This implies  $N_G(X)/C_G(X)$ is a $2$-group.

Thus by Thompson's normal $p$-complement theorem \cite[Theorem 2.27]{Suz}
$G = RP$.  Since $C_G(K) = K$,  $K$ acts without non-trivial fixed points on $R$.  Therefore if $x$ is a
non-central involution in $K$, the map which sends $r \in R$ to $rxr^{-1} \in Rx$ is injective.
Since $|R| = |Rx|$, every element of $Rx$ is of the form $rxr^{-1}$ for some $r \in R$.
But $(rxr^{-1})^2 = e$, so every element of $Rx$ is an involution.  This implies conjugation
by $x$ is inversion on $R$, so $R$ is abelian.

If $P=K$,  then $C_P(R) = Z(G)$, a cyclic subgroup of index $2$ in $P$.
Condition (\ref{nontrivial center 2}) then holds.

On the other hand, if $K$ is a proper subgroup of $P$,  then there are other elementary abelian
subgroups of order $4$,  and the same analysis applies to each of them.  It follows
that every involution in $P \smallsetminus{Z(P)}$ induces inversion on
$R$,  whence the product of any two such involutions centralizes $R$.
Hence condition (\ref{nontrivial center 2}) again holds.
This completes the proof.
\end{proof}

The remaining case to treat is that $Z(G)=1$.    We first point out
some  restrictions on $R$.

\begin{lemma} \label{restrictions trivial center}
Let $G$ be an
$\Oo$-group for $p=2$ with trivial center and with a non-cyclic Sylow $2$-subgroup $P$.  Set $R=O(G)$.
\begin{enumerate}
\item \label{nilp commutator}
$[R,R]$ is nilpotent.
\item \label{no A4}
If $A_4$ is not a subgroup of $G$, then $G = RP$ with $C_R(K)=1$, where $K \le P$ is elementary abelian of order $4$.
\end{enumerate}
\end{lemma}

\begin{proof}
For part (\ref{nilp commutator}) of the lemma, note that
$K$ acts on $R$ without fixed points other than the identity, by case (\ref{Oort 2 dihedral}) of Theorem \ref{Oort 2 Sylow}.  Now apply
\cite[Theorem 2]{bauman}.

For part (\ref{no A4}) of the lemma, we apply Thompson's normal $p$-complement
result as in the proof of Lemma~\ref{nontrivial center lemma}.
\end{proof}

We next consider more specifically the special case that  $R=1$.

\begin{lemma}  \label{nonsolv}
Under the hypotheses of Lemma~\ref{restrictions trivial center},
assume that $R=1$ and $A_4 \le G$.
Then one of the following holds:
\begin{enumerate}
\item $G=A_4$ or $S_4$;  or
\item $G=\PSL(2,q)$ or $\PGL(2,q)$ with $q$ an odd
prime power and $q > 4$.
\end{enumerate}
\end{lemma}

\begin{proof}
Let $A$ be a minimal normal subgroup of $G$.
Then $A$ has even order since $R=1$.  If $A$ is a $2$-group,
then it is elementary abelian.  Since $Z(G)=1$,  $|A| > 2$.  By
case (\ref{Oort 2 dihedral}) of Theorem \ref{Oort 2 Sylow}, one has $C_G(K) = K$
for every Klein four subgroup $K$ of $G$.  So $A$ must be elementary abelian of order $4$
and $A=C_G(A)$.  Therefore
$N_G(A)/A$ embeds in $\aut(A)  \cong \GL(2,2)  \cong S_3$
and is not a $2$-group since $Z(G) = 1$ and $G = N_G(A)$.  Thus,  $G/A \cong \Z/3$ or $S_3$.
It is trivial to see this is a split extension and
$G \cong A_4$ or $S_4$.

Now suppose, on the other hand, that no minimal normal subgroup of $G$
is a $2$-group. Then
$O_2(G)=1$,  and so  $F(G)=1$  (since $R=1$).
Thus $S:=F^*(G)=E(G)$
is a direct product of non-abelian simple groups.  Each factor of $E$
has even order.  Since $S$ is an Oort group at $2$, $S$ has a dihedral
$2$-Sylow subgroup, whence $S$ cannot be a non-trivial direct product.
 Thus, $S$ is simple with a dihedral $2$-Sylow subgroup.  By the classification of simple groups
with dihedral Sylow $2$-subgroups \cite[Theorem 1]{gw}, we
see that $S=\PSL(2,q)$ with $q$ odd and $q > 4$. Note that
since $F^*(G) = S$ we have $S \subset G \subset \mathrm{Aut}(S)$.

Let $P$ be a $2$-Sylow subgroup of $G$ and set $Q= P \cap S$.
Then $Q$ and $P$ are dihedral.  Since elementary abelian subgroups of order $4$
are self-centralizing by Theorem \ref{Oort 2 Sylow}, we see that $C_G(Q) \le Q$.   By Thompson's normal
$p$-complement theorem, or by inspection, $S$ contains a group isomorphic to $A_4$.

By Sylow's theorems, $G= \langle S, N_G(Q) \rangle$.
It follows that either $|Q|=4$ and
$N_G(Q) \cong A_4$, whence $N_G(P) \le S$ and so $G=S$,
or else $P$ is non-abelian and dihedral, whence $N_G(Q)=P$ (since the automorphism
group of $Q$ is a $2$-group).   Since $P$ is dihedral, it is generated by involutions,
as is $S$,
and so $G$ is generated by involutions in $\aut(S)$.

Recall that $\aut(S)$ is generated by $\PGL(2,q)$ and the Frobenius automorphism.
It follows by \cite[7.6]{gl} that any involution $x$ in $\aut(S)$ is either contained
in $\PGL_2(q)$ or $q=q_0^2$ and $x$ is conjugate to the $q_0$-Frobenius automorphism.
In the latter case, $C_S(x)$ contains a copy of $\PSL_2(q_0)$, whence $x$ centralizes
an elementary abelian subgroup of order $4$. Such an $x$ cannot be in $G$
by Theorem \ref{Oort 2 Sylow}.   Thus,
$\PSL(2,q) \le G \le \PGL(2,q)$ with $4 < q$ odd.
\end{proof}

\begin{lemma} \label{modules for PSL(2)}
Let $H= \PSL(2,q)$ or $\PGL(2,q)$ with $q$ odd and $q > 4$.  
Let $F$ be a field of characteristic $r \ne 2$, and let $V$
be an absolutely irreducible $FH$-module.  Suppose that there is 
a subgroup $J$ of $H$ isomorphic to $A_4$ 
such that there are no fixed points for the action of 
the Klein four subgroup $K: = O_2(J) \subset J$ on $V$.
Then the following hold:
\begin{enumerate}
\item \label{PSL2 dim3} $\dim V = 3$;
\item \label{PSL2 involution trace}  every involution $x$ in $J$ has trace $-1$ on $V$;
\item \label{PSL2 cases}  either $r|q$ or $H =\PSL(2,q),  q \le 7$.
\end{enumerate}
\end{lemma}

\begin{proof} Without loss of generality, we can base change to the algebraic
closure of $F$ to be able to assume that $F$ is algebraically closed.
 Let $\chi$ denote the Brauer character
of $H$ on $V$.  For all $r$, the three-dimensional irreducible representation
$M$ of $A_4$ over $F$ is projective and injective as a module for $FH$.
If $V$ is not the direct sum of copies
of $M$ (as an $A_4$-module), then the socle of $V$ contains a one-dimensional character of $J$, and
this character is trivial on $K$.  Since $K$ acts fixed point freely on
$V$,  this is contradiction.  So $V$ is a direct sum of copies of $M$, and  $\chi(x)  = -\dim V/3$
for each $x \in J$ of order $2$, because the restriction of $M$ to $K$ is the sum of the
three non-trivial one-dimensional characters of $M$.

By considering the structure of $V$ as an $A_4$-module, we see that $V$ is tensor
indecomposable.
 
First assume that $r|q$.   We view $V$ as a module for
$\SL(2,q)$.   By Steinberg's tensor product theorem
and the fact that $V$ is tensor indecomposable, we see that $V$
is a Frobenius twist of a module $L(d)$, where $L(d)$
can be identified with the space of  homogenous  polynomials in two variables of
degree $d$ for $0 \le d < r$.  So $\dim V = d +1$.  Since the central involution in $\SL(2,q)$ acts trivially,
$d$ is even.   It is easy to compute that the character of an involution in $\PSL(2,q)$ on
$L(d)$ lies in $\{-1,0,1\}$.  However, we know that $\chi(x) = -\dim V/3$ for at least one involution
$x$, namely one lying in $K$.  Therefore $\chi(x) \in \{-1,0,1\}$
implies that $d=2$, $V$ is $3$-dimensional and $\chi(x) = -1$.

Now suppose that $r$ does not divide $q$.  The modular
representations of $H$ are well known  \cite{burkhardt},
and indeed they are all  reductions of
characteristic zero  representations.  It follows
that $\chi(x)=-\dim V/3$ if and only if $\dim V =3$.
Since the smallest non-trivial irreducible representation of $\PSL(2,q)$
has dimension $(q-1)/2$, it follows that $q=5$ or $7$.   The smallest
non-trivial irreducible representation of $\PGL(2,q)$ is $q-1$, whence
$H = \PSL(2,q)$.
\end{proof}

\begin{remark} If $H=\PSL(2,q)$ and $V$ is an irreducible
module of dimension $3$,  then $V$ satisfies the hypotheses
of the previous lemma.
If $G=\PGL(2,q)$,  each $3$-dimensional module for $\PSL(2,q) $
such that $r|q$ has two extensions to $\PGL(2,q)$, and one
of those extensions will satisfy the hypotheses of the previous
lemma.
\end{remark}

\begin{corollary} \label{action}  Let $H=\PSL(2,q)$ with $q$ odd and
$q > 4$.  Let $R$ be a finite abelian group on which $H$ acts.  Suppose there
is a subgroup $J$ of $H$ isomorphic to $A_4$ whose Klein four subgroup $K$
acts fixed point freely  on the non-identity elements of $R$.
If $a \in R$ is non-trivial, then $K$ has a regular orbit on $Ha$.
\end{corollary}

\begin{proof}    Since $K$ acts fixed point freely on the non-identity elements of $R$, $R$ has odd order.
There is no harm in assuming that $R$ is generated as
a $\mathbb{Z}H$-module by $a$.   We can also pass to a quotient and
assume that $R$ is an irreducible $H$-module.   Let $E=\mathrm{End}_H(R)$.
Thus, $E$ is a field of characteristic $r > 2$, and we may view  $R$ as an absolutely irreducible
$EH$-module.

By Lemma~\ref{modules for PSL(2)},  $R$ is $3$-dimensional.  Let $L$ be the stabilizer of $a$ in $K$.
Since $K$ acts fixed point freely, $L \ne K$.  If $L=1$, the result follows.  So we may
assume that $|L|=2$.   Let $1 \ne x \in L$.   By part (\ref{PSL2 involution trace}) of Lemma \ref{modules for PSL(2)},
$x$ has a $1$-dimensional fixed space
on $R$.  This space must be the span of $a$ since $L$ stabilizes $a$.   So the stabilizer of the line $W$ generated by $a$
contains $C_H(x)$.  The subalgebra $A(x)$ of $\mathrm{Mat}_2(\mathbb{F}_q)$ which is generated 
over $\mathbb{F}_q$ by
an inverse image of $x$ in $\mathrm{SL}(2,q)$ is commutative and semi-simple of
dimension $2$ over $\mathbb{F}_q$.  Each inverse image in $\mathrm{SL}(2,q)$
of an element of $C_H(x)$ must conjugate $A(x)$ back to itself.  By considering the possible actions
of such elements on $A(x)$, one sees that $C_H(x)$  is the normalizer $N_H(T)$
of the maximal torus $T$ in $H$ which contains $x$, and $|N_H(T)| =  q \pm 1$.

If $r$ divides $q$, then (up to a Frobenius twist which does not affect
the result) $R$ is the representation of $H$ which results
from the action of $\SL(2,q)$ on quadratic polynomials, and $x \in H$ lifts
to an element $\tilde{x} \in \SL(2,q)$ having eigenvalues $\pm \sqrt{-1}$.

The representation $R$ of $H$ over $E$ is self-dual, so there is a non-degenerate $H$-invariant
quadratic form on $R$.   We can decompose $R$ into the direct sum of the eigenspaces of $x$,
and eigenspaces with different eigenvalues are orthogonal and stabilized by the action of $H$.
In particular, $W$ and its orthogonal complement are non-degenerate for the quadratic form.
The stabilizer in $H$ of the line $W$ through $a$ leaves invariant
the orthogonal complement of $W$.   Since $x$ acts as $-1$ on this complement, we see that
the stabilizer of the line containing $a$ is precisely $C_H(x) = N_H(T)$.  By considering
the cases in which $T$ is split or not split, and using the above description of the three
dimensional representation $R$, one sees that $T$ is the stabilizer of $a$, where $T$
has index $2$ in $C_H(x) = N_H(T)$.    Note that $Ha \cap W =\{ \pm a \}$ (since they are the only two
vectors in $W$ with the same norm with respect to the $H$-invariant quadratic form).   Since
every non-trivial element of $K$ is a conjugate of $x$ by an element of $J$, it follows that any non-trivial element of $K$ fixes exactly $2$ points
in $Ha$.  Therefore, as long as $|Ha| = [H:T] \ge q(q-1) > 6$, there will be $b \in Ha$ not fixed by any non-trivial element
of $K$, as required. 

We now suppose that $r$ does not divide $q$.  So $q=5$ or $7$ by Lemma \ref{modules for PSL(2)}.  If $q=5$, then $\PSL(2,5) \cong A_5$. All 
irreducible representations of $A_5$ are self dual (in any characteristic) and so precisely the same
argument as above applies.

Finally, assume that $q=7$.  Note that $\PSL(2,7) \cong \PSL(3,2)$ and that $S:=N_H(T)$ is a Sylow $2$-subgroup
of $H$ (of order $8$).  We claim that $S$ is precisely the stabilizer of $W$.  Since $S$ is a Borel
subgroup of $H$ (considered as $\PSL(3,2)$), the only proper overgroups of $S$ are the two parabolic subgroups
containing $S$ each isomorphic to $S_4$.  If $S_4$ preserves $W$, then its derived subgroup $A_4$ acts trivially
on $W$, a contradiction to $K$ having no fixed points.   So $S$ is the stabilizer of $W$ and contains the stabilizer of
$a$.  Let $S^h$ be the opposite Borel subgroup to $S$, so that $S^h$ is a conjugate of $S$.  We have $K \subset C_H(x) = S$.  Thus,  $K \cap S^h \le S \cap S^h =1$ and so $K$ is disjoint from
some stabilizer on the set $Ha$, whence $K$ has a regular orbit on $Ha$.
\end{proof}

\begin{corollary} \label{nonsolv2} Let $G$ be a non-solvable $\Oo$-group.
Let $R=O(G)$.  Then:
\begin{enumerate}
\item \label{nonsolv2 nilp} $R$ is nilpotent.
\item \label{nonsolv2 quotient}   $G/R$ is either
$\PSL(2,q)$ or $\PGL(2,q)$, with $q$ odd and $q > 4$.
\item \label{nonsolv2 chief}  Every chief factor $X = H/N$ of $G$ for which $H$ is contained in $R$
has the following properties.  The action of $R$ on $X$ is trivial, and  $X$ is an irreducible
$3$-dimensional $G/R$-module
(over the endomorphism ring of $X$ as $G/R$-module)
satisfying the conditions on $V$ in Lemma~\ref{modules for PSL(2)}.
 There are no fixed points for the action of $K$ on $X$.
\item  \label{nonsolv2 rgroup} If $q > 7$ or $G/R =\PGL(2,q)$,  then $R$ is an $r$-group where $r|q$.
\end{enumerate}
\end{corollary}

\begin{proof}  By Theorem \ref{Oort 2 Sylow}, the $2$-Sylow of $G/R$ is not
cyclic because $G/R$ is not solvable.  Now we apply Lemma \ref{nontrivial center lemma}
to conclude that $Z(G/R)$ is trivial because $G/R$ is not solvable.  Hence
$G/R$ satisfies the hypotheses of Lemma \ref{restrictions trivial center}.
We may therefore apply Lemma \ref{nonsolv} to $G/R$.  It is impossible
that $G/R$ is isomorphic to $A_4$ or $S_4$ since $G$ is not solvable. Hence
Lemma \ref{nonsolv} shows 
$G/R \cong \PSL(2,q) $ or $\PGL(2,q)$
with $q$ odd and $q > 4$.   The proves part (\ref{nonsolv2 quotient}) of Corollary \ref{nonsolv2}.

We now focus on proving part (\ref{nonsolv2 nilp}) of the corollary.
By Lemma 
\ref{restrictions trivial center} applied to $G/R$, there is an $A_4$ subgroup
$J$ of $G/R$.  Since $R$ is normal of odd order, there is an elementary abelian
subgroup $K$  of $G$ of order $4$ such that $KR/R$ is contained in $J$.  
By case (\ref{Oort 2 dihedral}) of Theorem \ref{Oort 2 Sylow}, $K$ acts fixed
point freely on $R$.  Since $G/R$ has trivial center, the center of $G$
is contained in $R$; but no non-trivial element of $R$ is fixed by $K$.
Hence $G$ has trivial center, so by applying part (\ref{nilp commutator}) of Lemma \ref{restrictions trivial center}
to $G$ we conclude that $M:=[R,R]$ is nilpotent.  If $M$ is trivial then $R$ is abelian
and hence (\ref{nonsolv2 nilp}) of the corollary holds.  Therefore we now assume $M$ is 
non-trivial, so that $Z(M)$ is non-trivial because $M$ is nilpotent.

Since $Z(M)$ is a non-trivial normal subgroup of $G$, there exists 
a minimal non-trivial normal subgroup $A$ of $G$ 
contained in $Z(M)$.  Then $G/A$ is an $\Oo$-group by Lemma \ref{lem:sections}. To
prove part (\ref{nonsolv2 nilp}) of the corollary, 
we may assume by induction
that $R/A$ is nilpotent.  Since $M$ centralizes $A$,  the action of $R$ on $A$ factors
through $R/[R,R]$, so $R$ acts on
$A$ as an abelian
group.   If $R$ centralizes $A$, then we see that $R$ is nilpotent by considering the ascending 
central series of $R$ together with the fact that $R/A$ is nilpotent.  Thus to prove that
(\ref{nonsolv2 nilp}) holds, it will suffice to derive a contradiction from the assumption that $R$ acts non-trivially on $A$.

View $A$ as $G$-module over the field $E$,  the $G$-endomorphism ring of $A$.
So $A$ is an absolutely irreducible $E[G]$-module with $K$ acting fixed
point freely on $A$.  We know that $R$ acts as abelian group on $A$, and
we are supposing this action is non-trivial.  Let $\ell$ be the characteristic of $E$.
The $\ell$-Sylow subgroup of $R/[R,R]$ fixes a non-trivial $E[G]$-submodule of
$A$.  So since $A$ is absolutely irreducible as an $E[G]$-module, it follows 
that the action of $R/[R,R]$ on $A$ factors through the maximal prime-to-$\ell$
quotient of $R/[R,R]$.  In particular, this action is semi-simple. 

Let $V = A \otimes_E k$ where $k$ is the algebraic closure of $E$.
Then $V$ is an irreducible $kG$-module.  Since $K$ acts fixed point freely
on $R$, the $E[K]$-module $A$ contains no non-trivial $K$-invariants, so the same
is true for the $k[K]$-module $V$.  Hence $K$ acts fixed point freely on
$V$, and we have shown that $R$ acts semi-simply as a non-trivial abelian group.     So $V$ is the direct
sum of $R$-eigenspaces on $V$ and $G$ permutes these eigenspaces
transitively.   Since the action of $R$ is non-trivial, there must be an eigenspace
associated to a non-trivial character $\lambda$ of the abelian group $R^\#:=\mathrm{Hom}(R,k^\times)$.  (Note that in fact, $R$ has no fixed points on $V$, since that would be an
invariant subspace).   Consider the action of $G$ on $R^\#$.    Of course,
$R \subset G$ acts trivially on this group and so we may view $\PSL(2,q) \subset G/R$ acting on $R^\#$.
We have assumed that $K$ acts fixed point freely on the non-trivial element of $R$.  Since $R$ is solvable
of order prime to $|K|$, the action of $K$ on the non-trivial elements of $R^\#$ is also fixed point free.  By
Corollary \ref{action},  $K$ has a regular orbit on $\PSL(2,q) \lambda$ and so on the
$R$-eigenspaces on $V$.  Thus $V$ contains a free $k[K]$-submodule, and so
$K$ has fixed points on $V$ which is a contradiction.   This shows
that $R$ centralizes $A$ and so $R$ is nilpotent and 
we have proved part (\ref{nonsolv2 nilp}) of the corollary.

Let $X = H/N$ be a chief factor of $G$ for which $H$ is contained in $R$.
Then $R$ acts trivially on $X$, since otherwise $[H,R]N$ would be
a normal subgroup of $G$ properly between $N$ and $H$.  Thus
$X$ is a $G/R$-module which must be irreducible.  We know that
$H \subset R$,
$K$ acts fixed point freely on $R$, $R$ is nilpotent and that $H$ is normal
in $G$.  It follows that $X$ can have no fixed points under the action of
$K$, since these would lift to fixed points for $K$ acting on $H$.
Now the hypothesis of Lemma  \ref{modules for PSL(2)} are satisfied
for the $G/R$-module $V = X$ when we let $F$ be the $G/R$-endomorphism
ring of $X$.  The remaining assertions (\ref{nonsolv2 chief}) and (\ref{nonsolv2 rgroup}) now 
follow by Lemma~\ref{modules for PSL(2)}.
\end{proof}

Using the above results, we can now prove Theorem \ref{main2}, completing the classification 
of $\Oo$-groups for the prime $2$.

\begin{proof}[Proof of Theorem \ref{main2}]
Let $G$ be as in the theorem.  Then
$[R,R]$ is nilpotent  by Lemmas \ref{nontrivial center lemma} and \ref {restrictions trivial center},
and these lemmas also show that $G = RP$ if $G$ has no subgroup isomorphic to $A_4$.
Every elementary abelian subgroup of order $4$ acts fixed point freely on 
the non-trivial elements of $R$  by case (\ref{Oort 2 dihedral}) of Theorem \ref{Oort 2 Sylow}.

For the remainder of the proof, we assume that $G$ has a subgroup isomorphic to $A_4$,
and we will show that either (\ref{A4 complement main2}), (\ref{S4 complement main2}), 
or (\ref{nonsolv main2}) of the theorem holds.  
Since $G$ is an $\Oo$-group, so is $G/R$, by Lemma~\ref{lem:sections}.  Lemma \ref{nonsolv} shows  $G/R$ is one of $A_4,  S_4$,  $\PSL(2,q)$,  or
$\PGL(2,q)$ with $q > 4$ and $q$ odd.  For the last two of those groups,
case (\ref{nonsolv main2}) of the theorem holds by Corollary \ref{nonsolv2}.  So we now assume that
$G/R \cong A_4$ or $S_4$, and show that (\ref{A4 complement main2}) or (\ref{S4 complement main2}) respectively holds.  

We now show  that $G$ has a complement to
$R$ of the form $N_G(K)$ for some Klein four subgroup $K$.  
First, since $G/R$ contains a normal elementary abelian subgroup of order $4$,
and since the order of $R$ is odd, Schur-Zassenhaus implies that 
$G$ contains an elementary abelian subgroup $K$ of order $4$
that is normal modulo $R$.  
It acts fixed point freely on 
the non-identity elements of $R$.  
If $g \in G$, then $RK^{g^{-1}}=RK$ by normality modulo $R$, and so $K$ and $K^{g^{-1}}$ are
Sylow $2$ subgroups of the group $RK$.  By Sylow's theorem, $K^{g^{-1}r} = K$
for some $r \in R$.  Hence $K=K^{r^{-1}g}$ and so $r^{-1}g \in N_G(K)$.  Thus 
$g = r(r^{-1}g) \in RN_G(K)$, proving $G=RN_G(K)$.
So to prove that $N_G(K)$ is a complement to $R$ in $G$, 
it suffices to show that if $ r \in N_R(K) = N_G(K) \cap R$ then $r$ 
is trivial.  Suppose $r$ is not trivial.  
Then $r$ and $K$
must generate a cyclic-by-$2$ $\Oo$-group $K\langle r\rangle$ which is not a $2$-group.  This forces $K\langle r\rangle$
to be isomorphic to $A_4$, so $r \in R$ has order $3$ and acts as a 
cyclic permutation of the non-trivial elements of $K$.  By $G = RN_G(K)$
and the fact that $G/R$ contains a subgroup isomorphic to $A_4$, we can
find an element $s \in N_G(K)$ whose image in $G/R$ has order $3$ and whose action on $K$
is the same as that of $r$.  Then $s r^{-1}$ is an element of $C_G(K)$ with
image of order $3$ in $G/R$, and this is impossible since $C_G(K) = K$
by case (\ref{Oort 2 dihedral}) of Theorem \ref{Oort 2 Sylow}.  This contradiction shows 
that in fact $N_R(K) = N_G(K) \cap R$
is trivial, so that $N_G(K)$ is a complement to $R$, proving the claim.  
Here this complement is isomorphic to $A_4$ or $S_4$ respectively, by the assumption
on $G/R$.

Consider any $G$-chief factor $X = H/N$ of $R$, so that $H$ and $N$
are normal subgroups of $G$ that are contained in $R$. Then $X$ has the properties
listed in part (\ref{nonsolv2 chief}) of Corollary \ref{nonsolv2} as a module for $G/R$, where $G/R \cong A_4$ or $S_4$. It is an elementary exercise to show that 
$A_4$ and $S_4$ have precisely one irreducible module over
$\mathbb{Z}/r$  in which
every elementary subgroup of order $4$ acts without fixed points.
It has order $r^3$ and has the properties stated in parts (\ref{A4 complement main2}) and (\ref{S4 complement main2}) of 
Theorem \ref{main2}.  This completes the proof.
\end{proof}

\end{document}